\documentclass[a4paper,12pt]{article}

\usepackage{amsmath,amssymb,amsfonts,a4wide}
\usepackage{amsthm}                   
\usepackage{graphics}                 
\usepackage{color}                    
\usepackage{hyperref}                 
\usepackage{epsfig}
\usepackage{graphicx}
\usepackage{psfrag}
\usepackage{fancyhdr}
\usepackage{setspace}
\usepackage{subfigure}
\newtheorem{theorem}{Theorem}

\newtheorem{assumption}[theorem]{Assumption}
\usepackage{color}
\usepackage{algorithm}

\newtheorem{definition}[theorem]{Definition}

{ 

\newtheorem{example}[theorem]{Example}
}

\newcommand{\bi}{\begin{itemize}}
\newcommand{\ei}{\end{itemize}}
\newcommand{\bd}{\begin{displaymath}}
\newcommand{\ed}{\end{displaymath}}
\newcommand{\be}{\begin{eqnarray*}}
\newcommand{\ee}{\end{eqnarray*}}

\title{\LARGE \bf
On Information Transfer in Control Dynamical Systems
}
\author{Subhrajit Sinha and  Umesh Vaidya\\
\thanks{Financial support from the National Science Foundation grant  ECCS-1150405 and CNS-1329915 is gratefully acknowledged. S. Sinha and U. Vaidya is with the Department of Electrical \& Computer Engineering,
Iowa State University, Ames, IA 50011
        {\tt\small ugvaidya@iastate.edu}}%
}

\begin{document}
\maketitle
\begin{abstract}
In this paper, we show through examples, how the existing definitions of information transfer, namely directed information and transfer entropy fail to capture true causal interaction between states in control dynamical system. We propose a new definition of information transfer, based on the ideas from dynamical system theory, and show that this new definition can capture true causal interaction between states. The information transfer measure is generalized to define transfers between the various signals in a control dynamical system and analytical expression for information transfer between state-to-state, input-to-state, state-to-output, and input-to-output are provided for linear systems. There is a natural extension of our proposed definition to define information transfer over $n$ time steps and average information transfer over infinite time step. We show that the average information transfer in feedback control system between plant output and input is equal to the entropy of the open loop dynamics thereby re-deriving the Bode fundamental limitation results using the proposed definition of transfer.
\end{abstract}

\section{Introduction}
From the days of Aristotle, philosophers and scientists have been concerned with the notion of causality. As far back as 300 B.C. Aristotle had realized the importance of causality and even  after about two and half thousand years, there is no universally accepted definition of causality. But causality and influence characterization is a subject of immense importance and finds its application in many different applications like world wide web and social media, biological networks, neural science, economics, finance etc. Usually, concepts of information theory are used in such applications and a study of the information flow between the components of the network throws light on causality and the influential nodes of the network. In \cite{IT_socialmedia} the authors use information based metric to characterize the most influential nodes in social networks. In neuroscience, concepts of information theory are used to understand how information flows in different parts of the brain \cite{IT_brain} and identifying influence in gene regulatory networks \cite{IT_bionetwork1,IT_bionetwork2} . In economic and financial networks, information transfer can be used to infer causal interactions from the time series data \cite{granger_economics, IT_economics,granger_causality}. In control theoretic setting, in \cite{seth2004information}, the authors have used information theoretic ideas to study the classical concepts of control like controllability and observability. 

Causality characterization was initially geared towards time series data and Granger causality \cite{granger_economics},\cite{granger_causality}, directed information \cite{IT_massey_directed},\cite{IT_kramer_directedit} and Schreiber's transfer entropy \cite{IT_schreiber} have been the most popular tools used for inferring the causality structure and influence characterization. In dynamical system setting, it was Liang and Kleeman \cite{liang_kleeman_prl},\cite{liang_predictability} who introduced the concept of information transfer between the states and used it for predictability analysis. 

The formulation of information transfer used in this paper is in dynamical systems setting and is closely related to and inspired from information transfer framework developed in \cite{Liang_Kleeman,liang_kleeman_prl,IT_pnas} for nonlinear dynamical systems. We used the ideas of Liang-Kleeman transfer to propose an axiomatic definition of information transfer in discrete linear dynamic network \cite{cdc_IT}. The axioms were physically motivated and it was shown in \cite{cdc_IT} that there exists a unique expression for information transfer in a dynamic network satisfying these axioms. In \cite{cdc_IT}, we had used absolute entropy to characterize the information flow. However, motivated by the definitions of directed information and transfer entropy, in this work we use conditional entropy instead of absolute entropy to characterize the information flow. 

Motivation for this work lies in the fact that, in dynamical systems, directed information fails to capture the intuitions of information transfer, and thus provide erroneous conclusion about the causal structure in a dynamical system. We show that our definition of information transfer does capture the correct causal structure of the system. We show, how this definition of information transfer between the states of a dynamical system can be easily extended to study the the information transfer between the inputs and outputs in a control dynamical system. In fact, this natural extension leads us to connect the information transfer to the Bode integral of the sensitivity transfer function from the output to the input in a feedback control system. 

The paper is organised as follows. In section \ref{section_DI}, we discuss, through some examples, how directed information fails to capture the notions of causality (zero transfer and indirect influence). In section \ref{section_IT} we provide our definition of information transfer. We also define $n$-step information transfer and average information transfer and for linear systems, we provide explicit formulas for computing the transfer. We also revisit the examples in section \ref{section_DI} and show how our definition of information transfer captures the intuitions of information transfer. In section \ref{section_IT_CDS}, we generalize our definition of information transfer to define information transfer between the inputs and outputs in a linear control dynamical system. We also study the information transfer in a feedback control system and show how this is related to the Bode integral of the sensitivity transfer function from the output to the input. This is followed by conclusions in section \ref{section_conclusion}.

\section{Directed Information as a Measure of Causality}\label{section_DI}
Directed information and transfer entropy are two of the most popular notions of information transfer used to characterize causality. In this section, we show that these notions of information transfer cannot faithfully capture the true causality structure in control dynamical system. 

Directed information was first defined by Massey \cite{IT_massey_directed} as a generalization of Marko's bidirectional information \cite{Marko}. Both bidirectional information and directed information gave a sense of direction to Shannon's information theory and is viewed as a generalized information theory.
Let $X^n=\{X_1,X_2,\cdots , X_n\}$ and $Y^n= \{Y_1,Y_2,\cdots , Y_n\}$ be two stochastic processes, viewed as a sequence of random variables. 
Massey and Kramer \cite{IT_kramer_directedit} defined the directed information from $X^n$ to $Y^n$ as
\begin{eqnarray}\label{directed_IT}
I(X^n\to Y^n) = H(Y^n) - H(Y^n\parallel X^n)
\end{eqnarray}

where $H(Y^n)$ is the entropy of the sequence $Y^n$ and $H(Y^n\parallel X^n) := \sum_{i = 1}^n H(Y_i|Y^{i-1},X^i)$ is the entropy of $Y^n$ \emph{causally} conditioned on $X^n$. 

The directed information is asymmetric and gives a directional sense to the information and defines a measure to determine the direction of information flow. Since the development of the concept of directed information, this has been used in many different applications, like determining the channel capacity of a communication channel. Moreover, for Gaussian variables, directed information and Granger causality \cite{granger_causality}, \cite{granger_economics} are equivalent \cite{costanzo2014survey}. This allows the concept of directed information to be used as a measure of causality and hence it has been used to infer about the causal structure of statistical processes.
In the following, we demonstrate using three different examples how the directed information fails to capture the true causality structure in dynamical systems setting. While the arguments are made in the context of directed information, similar conclusions can be drawn in the context of transfer entropy. We claim that at the heart of the problem is the manner in which causal conditioning is performed in both these definitions of information transfer. 


\subsection{Examples}

\begin{example}\label{example_zero_IT}
Consider the following linear system with output
\begin{eqnarray}\label{example_1_system}
\begin{pmatrix}x_1(t+1) \\ 
x_2(t+1) \\ 
x_3(t+1)
\end{pmatrix} &=& \begin{pmatrix} 0 & 0.5 & 0.5 \\ 
0 & 0 & 0.5 \\ 
0 & 1 & 0
\end{pmatrix} \begin{pmatrix}
x_1(t) \\ 
x_2(t) \\ 
x_3(t)
\end{pmatrix} + \sigma\xi(k)\nonumber\\
\end{eqnarray}

\begin{figure}
\centering
\subfigure[]{\includegraphics[scale=.75]{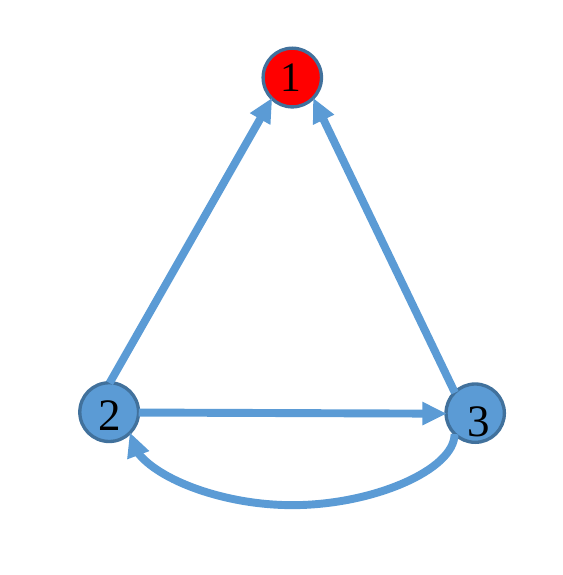}}\label{example_1_fig}
\subfigure[]{\includegraphics[scale=.75]{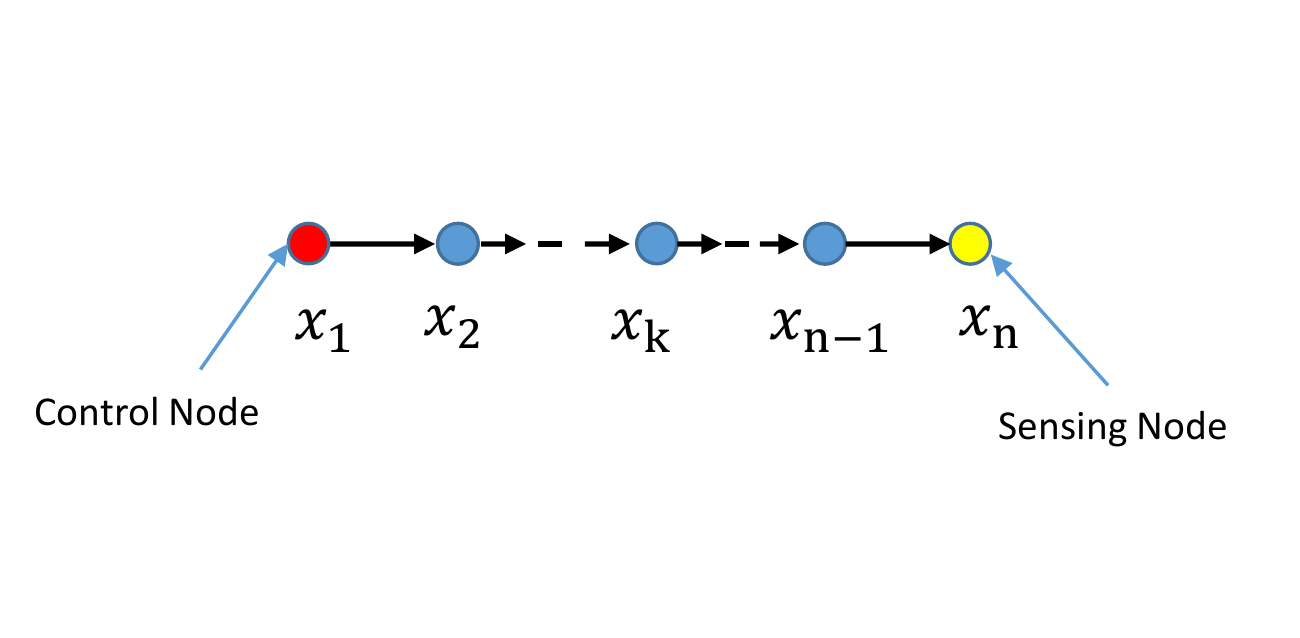}}\label{example_2_fig}
\caption{(a) Graph of dynamical system of example 1. (b) Graph of dynamical system of example 2.}
\end{figure}
where $x_i(t)$ are the states at time step $t$ and $\xi(t)$ is a independent identically distributed (i.i.d.) zero mean unit variance Gaussian noise and $\sigma$ is a constant. We notice that while there is directed path from $x_2\to x_1$ and $x_3\to x_1$ (fig. \ref{example_1_fig}(a)), there is no path from $x_1$ to $x_2$ or from $x_1$ to $x_3$. Hence, we conclude that $x_1$ is not a cause of $x_2$ or $x_3$, that is, $x_1$ is not influencing $x_2$ or $x_3$.  So if we treat the dynamical system as a stochastic process, we expect that the flow of information from $x_1\to x_2$ and $x_1\to x_3$ to be zero, thereby inferring that there is no causal connection from $x_1\to x_2$ and $x_1\to x_3$. In fact closer examination reveals that $x_1$ is not influencing $x_2$ and $x_3$ over any number of time steps. However, as we show below, the directed information fails to capture this true causal interaction between state $x_1,x_2,$ and $x_3$. In particular, we show that $I(x_1^n\to x_2^n)\neq 0$   and $I(x_1^n\to x_3^n)\neq 0$ for any $n$. In Fig. \ref{DI_3_node}, we plot the directed information from $x_1\to x_2$ and $x_1\to x_3$ over different time steps.

\begin{figure}[htp]
\centering
\subfigure[]{\includegraphics[scale=.34]{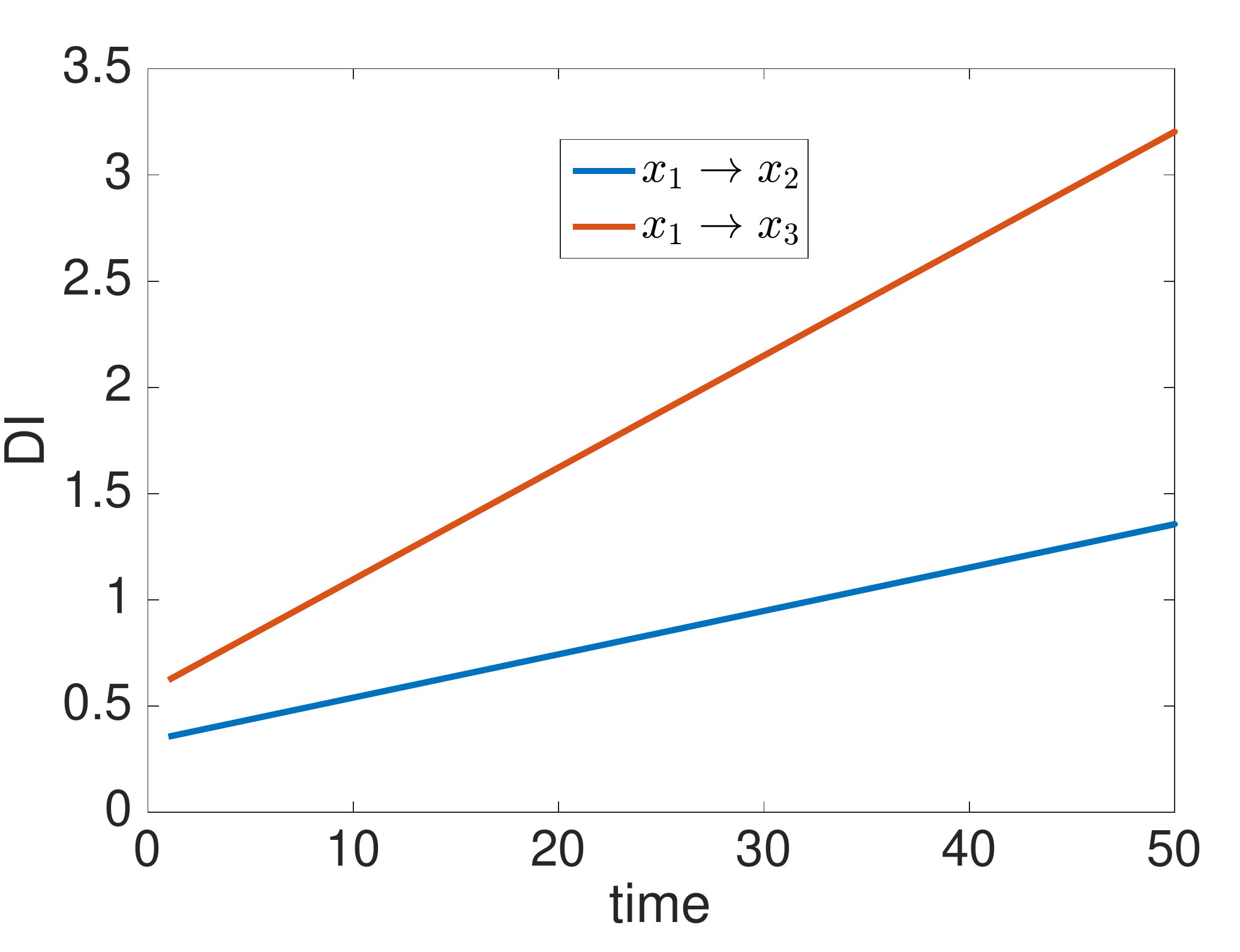}\label{DI_3_node}}
\subfigure[]{\includegraphics[scale=.34]{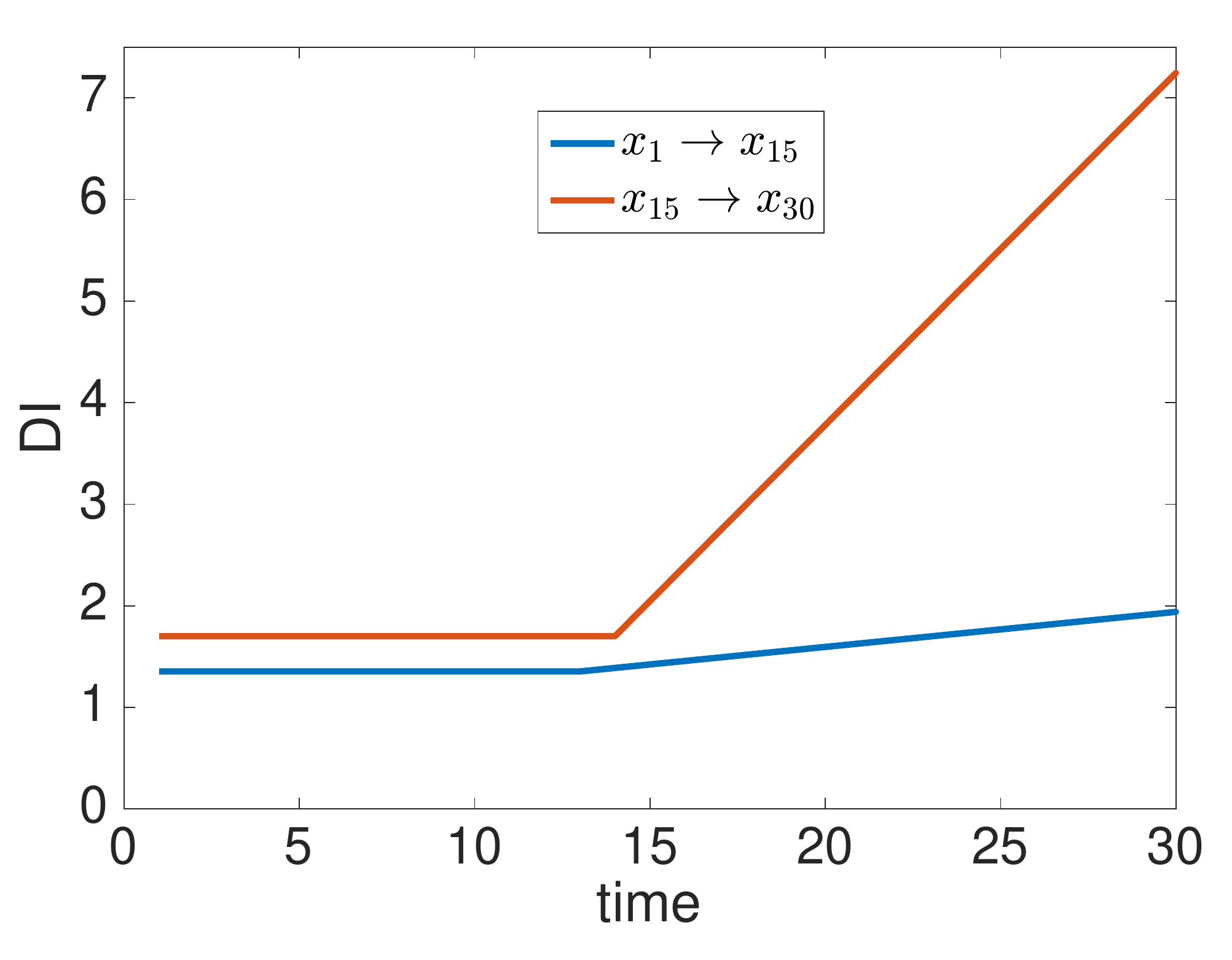}\label{DI_30_node}}
\caption{(a) Directed information plots for example 1. (b) Directed information plots for example 2.}
\end{figure}
\end{example}

\begin{example}\label{example_indirect}

Next we consider a single input single output linear system 
\begin{eqnarray}\nonumber\label{example_2_system}
&& x(t+1) = \begin{pmatrix}
0 & 0 & \hdots & 0 & 0 \\ 
1 & 0 & \hdots & 0 & 0\\
\vdots & \vdots & \vdots & \ddots & \vdots\\
0 & 0 & \hdots & 1 & 0
\end{pmatrix} x(t) + \begin{pmatrix}
1 \\
0\\
\vdots\\
0
\end{pmatrix} u\\
&& y(t) = \begin{pmatrix}0&\hdots& 1\end{pmatrix} x(t)
\end{eqnarray} 
where $x(t)\in \mathbb{R}^N$. We note that the state $x_1$ is directly controlled through $u$ and measurements are made at state $x_N$. State $x_k$, for $1<<k<<N$, is assumed to be some intermittent state. Since the control input $u$ does not affect state $x_k$ instantaneously, but through series of $k$ delays, we expect the information flow from state $x_1$ to $x_k$ to be zero for time step $t=0,\ldots, k-1$. On the other hand since $x_k$ influence $x_N$ again through series of delay we expect the information flow from $x_k$ to $x_n$ to zero for $t=k,\ldots, N-k-1$. However as we show in Fig. \ref{DI_30_node}, directed information from $x_1\to x_k$ and $x_k\to x_N$ is nonzero for  all time $t=0,\ldots,$ conveying that there is instantaneous flow of information from $x_1$ to $x_k$ and from $x_k$ to $x_N$. The simulation results in Fig. \ref{DI_30_node} are obtained for $N=30$ nodes with $k=15$. This example demonstrates that the directed information fails to capture indirect path of influence.

\end{example}

The above examples reveal some serious deficiencies of directed information as a measure of causality in dynamical system setting. This suggest that the directed information is a measure of {\it statistical} interaction between two signals but it does not successfully capture the true {\it dynamical} interactions between dynamical states.  
We expect that any meaningful definition of information transfer in network dynamical system be able to capture true causal interaction and also obey time constraint of information flow. Our proposed definition of information transfer discussed in the following section precisely does that.

\section{A New Definition of Information Transfer}\label{section_IT}
Consider the following discrete time dynamical system,
\begin{eqnarray}
z(t+1)=f(z(t))+\xi(t)\label{system}
\end{eqnarray}
where $z(t)\in \mathbb{R}^N$, $\xi(t)\in \mathbb{R}^N$ is assumed to vector valued random variable and $\xi(0),\xi(1),\ldots$ are independent random vectors each having the same density $g$. The mapping $f: \mathbb{R}^N\to \mathbb{R}^N$ is assumed to be at least continuous.  Let $z=(z_1,\ldots,z_N)^\top\in \mathbb{R}^N$. We are interested in defining the information transfer from state $z_i$ to state $z_j$, as the system evolves from time step $t$ to time step $t+1$. We denote this transfer by the notation $[T_{z_i\to z_j}]_t^{t+1}$. 


We provide the following definition for information transfer from state $z_i$ to state $z_j$ for the dynamical system (\ref{system}) :
\begin{definition}[Information transfer]\label{IT_def} The information transfer from $z_i$ to $z_j$ for dynamical system (\ref{system}) as the system evolves from time $t$ to time $t+1$ and denoted by $[T_{z_i\to z_j}]_t^{t+1}$ is given by following formula
\begin{eqnarray}\label{IT}
[T_{z_i\to z_j}]_t^{t+1}=H(\rho(z_j(t+1)|z_j(t)))-H(\rho_{\not{z}_i}(z_j(t+1)|z_j(t))
\end{eqnarray}
where $H(\rho(z_j))=- \int_{\mathbb{R}^{|z_j|}} \rho(z_j)\log \rho(z_j)dz_j$ is the entropy of probability density function $\rho(z_j)$ and $H(\rho_{\not{z}_i}(z_j(t+1)|z_j(t))$ is the entropy of $z_j(t+1)$, conditioned on $z_j(t)$, where $z_i$ has been frozen. 
\end{definition}
The above definition of information transfer can be understood by rewriting the expression of information transfer as follows:
\begin{eqnarray}
H(\rho(z_j(t+1)|z_j(t)))=[T_{z_i\to z_j}]_t^{t+1}+H(\rho_{\not{z}_i}(z_j(t+1)|z_j(t)))\label{transfer}
\end{eqnarray}
so that the total entropy of $z_j$ is the sum of transfer from $z_i$ and the entropy of $z_j$, when $z_i$ is absent.
%
%

\subsection{n-step Information Transfer}
The information transfer defined in (\ref{IT}) gives the information transferred from $z_i$ to $z_j$ as the dynamical system evolves from time step $t$ to time step $t+1$. So one can look at this definition as a \emph{one}-step transfer. To generalize the one step transfer to $n$-step transfer, we define the $n$-step transfer from $z_i$ to $z_j$ as the system evolves from time step $t$ to $t+n$, where $n\in\mathbb{Z}_{>0}$, as follows :
\begin{definition}[$n$-step Information transfer] The information transferred over $n$ time steps, $n\in\mathbb{Z}_{>0}$, from $z_i$ to $z_j$, which is denoted as $(T_{z_i\to z_j})_t^{t+n}$, as the dynamical system (\ref{system}) evolves from time step $t$ to $t+n$, is defined as
\begin{eqnarray}\label{n_stepIT}\nonumber
&&(T_{z_i\to z_j})_t^{t+n} =  H(\rho(z_j(t+n)|z_j(t+n-1)\cdots z_j(t)))\\
&& \qquad \qquad \qquad - H(\rho_{\not{z}_i}(z_j(t+n)|z_j(t+n-1)\cdots z_j(t)))
\end{eqnarray}
where $H(\rho(z_j(t+n)|z_j(t+n-1)\cdots z_j(t)))$ is the conditional entropy of $z_j(t+n)$, conditioned on past $n$-steps of $z_j$ and $H(\rho_{\not{z}_i}(z_j(t+n)|z_j(t+n-1)\cdots z_j(t)))$ is the conditional entropy of $z_j(t+n)$ on its past $n$-steps and the state $z_i$ has been frozen (held constant) from time step $t$ to time step $t+n$.
\end{definition}

In general, we have the following

\begin{theorem}\label{n_step_IT}
\begin{eqnarray}\label{sum_IT_n_step}
\sum_{i=1}^n(T_{z_i\to z_j})_t^{t+i} = H(z_j(t+n)\cdots z_j(t)) - H_{\not{z}_i}(z_j(t+n)\cdots z_j(t))
\end{eqnarray}
where 
\begin{eqnarray*}
H(z_j(t+n)\cdots z_j(t)) = H(\rho(z_j(t+n)\cdots z_j(t)))
\end{eqnarray*}
and 
\begin{eqnarray*}
H_{\not{z}_i}(z_j(t+n)\cdots z_j(t))=H(\rho_{\not{z}_i}(z_j(t+n)\cdots z_j(t))).
\end{eqnarray*}
\end{theorem}

\begin{proof}

We prove this by induction. For $i=1$, using the fact that $H(y)=H_{\not{x}}(y)$, we have $(T_{x\to y})_t^{t+1} = H(y',y)-H_{\not{x}}(y',y)$. 

Let this be true for $i=k$. Hence, we have,
\begin{eqnarray*}
\sum_{i=1}^k(T_{x\to y})_t^{t+i} = H(y(t+k)\cdots y(t))- H_{\not{x}}(y(t+k)\cdots y(t))
\end{eqnarray*}
Now,
\begin{eqnarray*}
\begin{aligned}
&(T_{x\to y})_t^{t+k+1} = H(y(t+k+1)|y(t+k)\cdots y(t)) - H_{\not{x}}(y(t+k+1)|y(t+k)\cdots y(t))\\
&= [H(y(t+k+1)\cdots y(t))-H_{\not{x}}(y(t+k+1)\cdots y(t))]\\
& \quad - [H(y(t+k)\cdots y(t))-H_{\not{x}}(y(t+k)\cdots y(t))]\\
&=  [H(y(t+k+1)\cdots y(t))-H_{\not{x}}(y(t+k+1)\cdots y(t))]- \sum_{i=1}^k(T_{x\to y})_t^{t+i} 
\end{aligned}
\end{eqnarray*}
Hence,
\begin{eqnarray*}
\sum_{i=1}^{k+1}(T_{x\to y})_t^{t+i} = H(y(t+k+1)\cdots y(t)) - H_{\not{x}}(y(t+k+1)\cdots y(t))
\end{eqnarray*}
Hence,

\begin{eqnarray*}
\sum_{i=1}^n(T_{x\to y})_t^{t+i} = H(y(t+n)\cdots y(t)) - H_{\not{x}}(y(t+n)\cdots y(t))
\end{eqnarray*}

\end{proof}

After this point, for notational convenience, we will use the notation $\rho(y_t^{t+n})$ to denote the density of the joint distribution of $y(t+n)y(t+n-1)\cdots y(t)$, that is,
$\rho(y_t^{t+n}) = \rho(y(t+n)y(t+n-1)\cdots y(t))$. When $t=0$, we will simply use $\rho(y^n)$ to denote $\rho(y(n)y(n-1)\cdots y(0))$.

\begin{definition}[Average Information Transfer] The average information transfer from $X$ to $Y$ is defined as  
\begin{eqnarray}\label{avg_IT}
\bar{T}_{x\to y}=\lim_{T\to \infty} \frac{1}{T}\sum_{i=0}^T(T_{x\to y})_0^{i}
\end{eqnarray}
\end{definition}

With this, 
\begin{eqnarray}\label{entropy_rate}
\bar{T}_{x\to y} = \lim_{n\to \infty}\frac{1}{n}\left[H(y^n) - H_{\not{x}}(y^n)\right]
\end{eqnarray}
Hence, the average information transfer from $X$ to $Y$ is a difference of two entropy rates, namely, entropy rate of $Y$ $(\lim_{n\to \infty}\frac{1}{n}H(y^n))$ and entropy rate of $Y$ when $X$ is frozen $(\lim_{n\to \infty}\frac{1}{n}H_{\not{x}}(y^n))$.

The average information transfer and average directed information, however, coincide as time goes to infinity. Intuitively, this can be explained as follows. Both average information transfer and average directed information can be written as a difference of two terms, where the first term is the same for both of them. The difference lies in the fact that, our definition of information transfer uses freezing, whereas, in directed information, one uses conditioning. However, as time goes to infinity, infinite step freezing and infinite step conditioning is the same  and hence they converge to the same value.
\begin{figure}[htp]
\centering
\includegraphics[scale=.4]{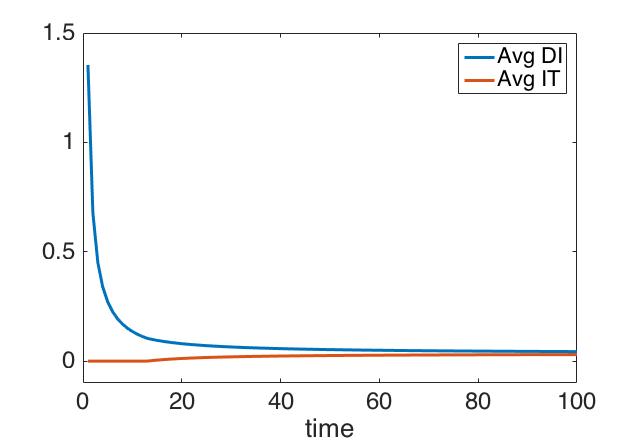}
\caption{Comparison of average information transfer and average directed information for example (\ref{example_indirect}).}\label{avg_IT_avg_DI}
\end{figure}
For example, lets consider the example (\ref{example_indirect}). In figure 3, we plot the average information transfer and average directed information from $x_1$ to $x_{30}$. From the figure, we see that though average information transfer and average directed information differs in the initial stages, as the system evolves, they both converge to the same value.

\subsection{Information transfer in linear dynamical system}
The Eq. (\ref{IT}) provides  formula for information transfer in general nonlinear system however for a special class of linear system analytical expression for the information transfer can be obtained. 
 Consider the following linear time invariant dynamical system 
\begin{eqnarray}
z(t+1)=Az(t)+\sigma \xi(t)\label{lti}
\end{eqnarray}
where $z(t)\in \mathbb{R}^N$ and $\xi(t)$ is vector valued Gaussian random variable with zero mean and unit variance. We assume that the initial conditions are Gaussian with covariance $\Sigma(0)$. Since the system is linear, the distribution of the system state for all future time will remain Gaussian with covariance $\Sigma(t)$ satisfying 
\[A \Sigma(t-1)A^\top+\sigma^2 I=\Sigma(t)\]
To define the information transfer between various subspace we introduce following notation to split the $A$ matrix :
\begin{eqnarray}
z(t+1)=\begin{pmatrix}x^{'}\\y^{'}\end{pmatrix}=\begin{pmatrix}A_x&A_{xy}\\ A_{yx}&A_{y}\end{pmatrix}\begin{pmatrix}x\\y\end{pmatrix}+\sigma \xi\label{splittingxy}
\end{eqnarray}

The $A$ matrix can be further split using the subspace decomposition $x=(x_1^\top,x_2^\top)^\top$ as follows:
\begin{eqnarray}
\begin{pmatrix}A_x&A_{xy}\\ A_{yx}&A_{y}\end{pmatrix}=\begin{pmatrix}A_{x_1}&A_{x_1x_2}& A_{x_1 y}\\A_{x_2x_1}&A_{x_2}& A_{x_2 y}\\ A_{y x_1}&A_{y x_2}& A_{y}\end{pmatrix}\label{splittingA}
\end{eqnarray}
Based on the decomposition of the system $A$ matrix we can also decompose the covariance matrix $\Sigma$ at time instant $t$ as follows. 
\begin{eqnarray}
\Sigma=\begin{pmatrix}\Sigma_x&\Sigma_{xy}\\\Sigma_{xy}^\top& \Sigma_y\end{pmatrix}=\begin{pmatrix} \Sigma_{x_1}&\Sigma_{x_1x_2}&\Sigma_{x_1 y}\\\Sigma_{x_1x_2}^\top&\Sigma_{x_2}&\Sigma_{x_2 y}\\\Sigma_{x_1y}^\top&\Sigma_{x_2y}^\top&\Sigma_{y}\end{pmatrix}\nonumber\\
\label{sigma_dec}
\end{eqnarray}
Using the above notation, we state following theorem providing explicit expression for information transfer in linear dynamical system during transient and steady state.

\begin{theorem}
Consider the linear dynamical system (\ref{lti}) and associated splitting of state space in Eqs. (\ref{splittingxy}) and (\ref{splittingA}). We have following expression for information transfer between various subspace
\begin{eqnarray}
[T_{x\to y}]_t^{t+1}=\frac{1}{2}\log \frac{|A_{yx}\Sigma^s_{y}(t)A_{yx}^\top +\sigma^2 I|}{\sigma^2}
\end{eqnarray}
where $\Sigma^s_y(t)=\Sigma_x(t)-\Sigma_{xy}(t)\Sigma_y(t)^{-1}\Sigma_{xy}(t)^\top$ is the Schur complement of $\Sigma_{y}(t)$ in the matrix $\Sigma(t)$. 

\begin{eqnarray}
[T_{x_1\to y}]_t^{t+1}=\frac{1}{2}\log \frac{|A_{yx}\Sigma^s_y(t)A_{yx}^\top +\sigma^2 I |}{|A_{yx_2}(\Sigma_y^{s})_{yx_2}(t)A_{yx_2}^\top+\sigma^2 I|}\label{transferx1y}
\end{eqnarray}

where $|\cdot|$ is the determinant and $ (\Sigma_y^s)_{yx_2}$ is the Schur complement of $\Sigma_{y}$ in the matrix 
\[\begin{pmatrix}\Sigma_{x_2}&\Sigma_{x_2y}\\\Sigma_{x_2 y}^\top&\Sigma_y\end{pmatrix}\]

\end{theorem}
\begin{proof}
From the formula of information transfer  we have
\begin{eqnarray*}
[T_{x\to y}]_t^{t+1}=H(\rho(y^{'}|y))-H_{\not{x}}(\rho(y^{'}|y))
\end{eqnarray*}
So to compute the transfer, we need to know $\rho(y^{'}|y)=\frac{\rho(y^{'},y)}{\rho(y)}$ and $\rho_{\not{x}}(y^{'}|y)$. Furthermore, since all the probability density function involved are Gaussian, we use following formula for the entropy of Gaussian density function.
\[H(\rho(w))=\frac{1}{2}\log ((2\pi e)^N |\Sigma(t)|)\]
for $\rho(w)$ of the form
\[\rho(w)=\frac{1}{(2\pi)^{\frac{N}{2}} |\Sigma(t)|^{\frac{1}{2}}}\exp\left(-\frac{1}{2} w^\top \Sigma(t)^{-1}z\right)\]
Hence to compute the required transfer, we only need to know the covariance of $\rho(y^{'}|y)$ and $\rho_{\not{x}}(y^{'}|y)$. The covariance matrix for the joint $z^{'}$ and $z$ is given by 
\[R_{z^{'}z}=\begin{pmatrix}\Sigma_{z^{'}}&\Sigma_{z^{'}z}\\\Sigma_{zz^{'}}& \Sigma_{z}\end{pmatrix},\;\;\Sigma_{z}=\begin{pmatrix}\Sigma_x&\Sigma_{xy}\\\Sigma_{xy}^\top& \Sigma_y\end{pmatrix}\]
where $\Sigma_{z^{'}}=A\Sigma_{z}A^\top+\sigma^2 I$, $\Sigma_{zz^{'}}=\Sigma_z A^\top$, $\Sigma_{z^{'}z}=A\Sigma_z=\Sigma_{zz^{'}}^\top$ 
Hence, we have 
\[R_{y^{'}y}=\begin{pmatrix}\Sigma_{y^{'}}&\Sigma_{y^{'}y}\\\Sigma_{y^{'}y}^\top& \Sigma_{y}\end{pmatrix},y^{'}|y\sim {\cal N}(-, \Sigma_{y^{'}}-\Sigma_{y^{'}y}\Sigma_y^{-1}\Sigma_{y^{'}y}^\top)\]
\begin{eqnarray*}&\Sigma_{y^{'}}=A_{yx}\Sigma_xA_{yx}^\top+A_{y}\Sigma_{xy}A_{yx}^\top+A_{yx}\Sigma_{xy}A_{y}^\top\nonumber\\&+A_{y}\Sigma_yA_{y}^\top+\sigma^2 I
\end{eqnarray*} \[\Sigma_{y^{'}y}=A_{yx}\Sigma_{xy}+A_y\Sigma_y\]
Note that since the entropy is function of covariance, we ignore the explicit computation of mean for the underlying Gaussian random variable. 
After simplification, we obtain 
\[\rho(y^{'}|y)\sim {\cal N}(-, A_{yx}(\Sigma_x-\Sigma_{xy}\Sigma_y^{-1}\Sigma_{xy}^\top)A_{yx}^\top+\sigma^2 I).\]
For $\rho_{\not{x}}(y^{'}|y)$ we look at the system, $
y^{'} = A_y y + \sigma \xi_y$.
Hence, we have $\rho_{\not{x}}(y^{'}|y) \sim{\cal N}\left(-, {\sigma^2}I\right)$.
Hence we have 
\[T_{x\to y}(t)=\frac{1}{2}\log \frac{|A_{yx}\Sigma^s_y(t)A_{yx}^\top +\sigma^2 I|}{\sigma^2}\]
where $\Sigma^s_y(t)=\Sigma_x(t)-\Sigma_{xy}(t)\Sigma_y(t)^{-1}\Sigma_{xy}(t)^\top$. 

To compute the information transfer from $x_1\to y$, we need to compute covariance for $\rho_{\not{x}_1}(y^{'}|y)$. Hence, we look at the system where $x_1$ is absent. Hence, we look at
\begin{eqnarray}
&&x_2^{'} = A_{x_2}x_2 + A_{x_2y} y + \xi_{x_2}\nonumber\\
&&y^{'} = A_{yx_2}x_2 + A_{y} y + \xi_{y}\label{yx1}
\end{eqnarray}
The expression for $\rho(y^{'}|y)$ remains unchanged from the previous case. To compute  $\rho_{\not x_1}(y^{'}|y)$, we use system equation (\ref{yx1}) and use procedure similar to one used in the derivation of $\rho_{\not x}(y^{'}|y)$ with $x$ replaced with $x_1$. Hence, we have
\begin{eqnarray}
\rho_{\not{x_1}}(y^{'}|y) = \mathcal{N}\left(-,A_{yx_2}(\Sigma_y^s)_{yx_2}(t)A_{yx_2}^\top + \sigma^2 I\right)
\end{eqnarray}
where, 
\begin{eqnarray*}
(\Sigma_y^s)_{yx_2}(t) = (\Sigma_{x_2}(t) - \Sigma_{x_2y}(t)\Sigma_y^{-1}(t)\Sigma_{x_2y}^\top (t)).
\end{eqnarray*}
Hence,
\begin{eqnarray}
[T_{x_1\to y}]_t^{t+1}=\frac{1}{2}\log \frac{|A_{yx}\Sigma^s_y(t)A_{yx}^\top +\sigma^2 I |}{|A_{yx_2}(\Sigma_y^{s})_{yx_2}(t)A_{yx_2}^\top+\sigma^2 I|}
\end{eqnarray}.
\end{proof}
The results from the above theorem can be used to provide general expression for information transfer between scalar state $z_i$ and $z_j$ for linear network system. In particular with no loss of generality we can assume $z_i=z_1$ and $z_j=z_2$, then the expression for  $T_{z_1\to z_2}$ can be obtained from (\ref{transferx1y}) by defining 
\[x_1:=z_1,\;\;\; y=z_2, \;\;x_2:=(z_3,\ldots,z_N)\] 

The formula for $n$-step transfer can also be derived in a similar manner, where one has to look at the $n$-step covariance matrix.

For linear systems with Gaussian noise, the one step zero transfer can be characterized by looking at system matrix $A$. In particular, we have the following theorem.

\begin{theorem}
 $A_{z_jz_i}=0$, iff $[T_{z_i\to z_j}]_t^{t+1}= 0$.
\end{theorem}

\begin{proof} The proof follows from formula (\ref{transferx1y}).

\end{proof}

\subsection{Information transfer and Causal Inference in Nonlinear Systems}
In this subsection we connect information transfer and causal inference for general nolinear systems. However, we will consider systems without noise. In particular, we consider the system 
\begin{eqnarray}\label{nonlinear_sys}
z_{t+1} = T(z_{t})
\end{eqnarray}
where $T : \mathbb{R}^n \to \mathbb{R}^n$ is assumed to be at least continuous and $z(t)\in \mathbb{R}^n$.
In component form, equation (\ref{nonlinear_sys}) can be written as
\begin{eqnarray}\label{nonlinear_sys_component}
\begin{aligned}
 z_{t+1}^1 & = T_1(z_t^1,z_t^2, \cdots , z_t^n)\\
 z_{t+1}^2 & = T_2(z_t^1,z_t^2, \cdots , z_t^n)\\
& \vdots \\
 z_{t+1}^n & = T_n(z_t^1,z_t^2, \cdots , z_t^n)\\
\end{aligned}
\end{eqnarray}
Information transfer between the states of a dynamical system is defined in terms of difference of certain entropies, as the dynamical system evolves in time. As such to compute the information transfer, it is necessary to know the evolution of densities. This is given by the Perron-Frobenius (PF) operator 
\begin{eqnarray}
\mathbb{P} : \mathcal{L}^1(\mathbb{R}^n) \to \mathcal{L}^1(\mathbb{R}^n)
\end{eqnarray}
and is defined as
\begin{eqnarray}
\int_{\omega}\mathbb{P}\rho(z)dz = \int_{T^{-1}(\omega)}\rho(z)dz 
\end{eqnarray}
for any $\omega\in\mathcal{B}(\mathbb{R}^n)$, where $\mathcal{B}(\mathbb{R}^n)$ is the Borel $\sigma$-algebra on $\mathbb{R}^n$. When $T$ is non-singular and invertible, the PF operator can be written explicitly as 
\begin{eqnarray}
\mathbb{P}\rho(z) = \rho [T^{-1}(z)]|J^{-1}|
\end{eqnarray}
where $J$ is the Jacobian of $T$ and $|\cdot |$ signify the determinant. The joint density $\rho(z)$ is used to define the Shannon entropy
\begin{eqnarray}
H(z) = - \int_\Omega\rho(z)\log \rho(z)dz
\end{eqnarray}
where $\Omega$ is the sample space.
\begin{assumption}\label{sample_space_assumption}
We assume $\Omega$ to be Cartesian product of $\Omega_1,\Omega_2, \cdots ,\Omega_n$, in which $z^1,z^2,\cdots , z^n$ are respectively lying. Here each $\Omega_i$ are open in $\mathbb{R}$, for $i = 1,2, \cdots , n$.
\end{assumption}

The marginal entropy of any variable of interest, say $z_1$ is defined as
\begin{eqnarray}
H(z^1) = -\int_{\Omega_1}\rho_1(z^1)\log\rho_1(z^1)dz^1
\end{eqnarray} 
where $\rho_1 = \rho(z^1) = \int_{\Omega_{2n}}\rho(z)dz^2dz^3\cdots dz^n$ is the marginal distribution of $z^1$. Here we use the notation
\begin{eqnarray}\label{notation_omega}
\Omega_{jn} = \Omega_j\times\Omega_{j+1}\times \cdots \times \Omega_n, \quad 1\leq j\leq n, \textnormal{ and } \Omega_{1n} = \Omega.
\end{eqnarray}
We further use the notation
\begin{eqnarray*}
&&\rho_{\not{i}} = \rho(z^1,\cdots z^{i-1},z^{i+1},\cdots , z^n)\\
&&= \int_{\Omega_i}\rho(z^1,\cdots z^{i-1},z^{i+1},\cdots , z^n)dz^i\\
&&\rho_{\not{i}\not{j}} = \rho(z^1,\cdots z^{i-1},z^{i+1},\cdots , z^{j-1},z^{j+1}\cdots ,z^n)\\
&&=\int_{\Omega_i\times\Omega_j}\rho(z^1,\cdots z^{i-1},z^{i+1},\cdots , z^{j-1},z^{j+1},\cdots , z^n)dz^idz^j
\end{eqnarray*}

With this, we state the following theorem which connects information transfer with causal inference.

\begin{theorem}
If $T_j$ is independent of $z^i$, then $T_{z^i\to z^j}=0$. At the same time, $T_{z^i\to z^j}$ need not be zero if $T_j$ does not rely on $z^i$.
\end{theorem}

\begin{proof}
Without loss of generality, we assume $i=2$ and $j=1$, so that we look at the information transfer from $z^2$ to $z^1$. We consider the system $z_{t+1} = T(z_t)$ given in component form as in (\ref{nonlinear_sys_component}). Define $\hat{z}_t^i = z_{t+1}^i$ and consider the dynamical system
\begin{eqnarray}\label{sys_extended_z}
\begin{aligned}
z_{t+1}^i &=& T_i(z)\\
\hat{z}_{t+1}^i &=& T_i(\hat{z})
\end{aligned}
\end{eqnarray}
for $i = 1,2,\cdots , n$. 
Now, from (\ref{info_trnasfer_eqv}),
\begin{eqnarray}\label{info_z_hat}\nonumber
T_{z^2\to z^1} &=& H(z_{t+1}^1,z_t^1) - H_{\not{z}_2}(z_{t+1}^1,z_t^1)\\
&=& H(z_{t+1}^1,\hat{z}_{t+1}^1) - H_{\not{z}_2}(z_{t+1}^1,\hat{z}_{t+1}^1)
\end{eqnarray}

It is given that $T_1(z)$ is independent of $z^2$. Hence, $T_1(\hat{z})$ is independent of $\hat{z}^2$ and also $z^2$. To prove that if $T_1(z)$ is independent of $z_2$, then $T_{z_2\to z_1}=0$, we have to show
\begin{eqnarray*}
H(z_{t+1}^1,\hat{z}_{t+1}^1) = H_{\not{z}_2}(z_{t+1}^1,\hat{z}_{t+1}^1)
\end{eqnarray*}

Let $U$ be the sample space where the system (\ref{sys_extended_z}) evolves. Note that $U$ satisfies assumption (\ref{sample_space_assumption}) and we use the notation defined in (\ref{notation_omega}). For notational convenience we use the notation $x_t^1$ to signify the subspace $(z_t^1,\hat{z}_t^1)^\top$, $x_t^2 = z_t^2$, $y_t = \hat{z}_t^2$, $x_t^i = (z_t^i,\hat{z}_t^i)^\top$ for $i = 3,\cdots , n$. Similarly we define $U_i$ and $U_y$. Hence we need to prove 
\begin{eqnarray*}
H(x_{t+1}^1) = H_{\not{x}_2}(x_{t+1}^1)
\end{eqnarray*}
The dynamical system (\ref{sys_extended_z}) can be written as
\begin{eqnarray}\label{sys_extended_compact}
\begin{aligned}
& x_{t+1}^i = \phi_i(x_t,y_t)\\
& x_{t+1}^2 = \phi_2(x_t,y_t)\\
& y_{t+1} = \phi_y(x_t,y_t)
\end{aligned}
\end{eqnarray}
for $i = 1, 3, \cdots , n$ and $\phi_i = (T_i,T_i)^\top$, $\phi_2 = T_2$ and $\phi_y=T_2$. Define $\Phi = \begin{pmatrix}
\phi_i^\top & \phi_2 & \phi_y
\end{pmatrix}^\top$.
For any subset of $\tilde{U}_1$, $\tilde{u}_1\in \tilde{U}_1$, we have
\begin{eqnarray*}
\int_{\tilde{u}_1}(\mathbb{P}\rho_{\not{x}^2})_1(x^1)dx^1 &=& \int_{\tilde{u}_1\times  U_{3n}\times U_y}(\mathbb{P}\rho_{\not{x}_2})(x^1,y,x^3,\cdots ,x^n)\\
&& \times dx^1 dx^3\cdots dx^n dy\\
&=& \int_{\Phi_{\not{2}}^{-1}(\tilde{u}_1\times  U_{3n}\times U_y)}\rho_{\not{x}_2}(x^1,y,x^3,\cdots ,x^n)\\
&& \times dx^1 dx^3\cdots dx^n dy\\
\end{eqnarray*}
Now\footnote{In the proof, $\Phi_{1\not{2}}^{-1}\tilde{u}_1$ and $\phi_1^{-1}\tilde{u}_1$ is the projection of the Cartesian product on the subspace $x^1$.},
\begin{eqnarray}
{\Phi_{\not{2}}^{-1}(\tilde{u}_1\times  U_{3n}\times U_y)} = \Phi_{1\not{2}}^{-1}\tilde{u}_1\times U_{3n}\times U_y
\end{eqnarray}

Hence, 
\begin{eqnarray*}
\int_{\tilde{u}_1}(\mathbb{P}\rho_{\not{x}^2})_1(x^1)dx^1 &=& \int_{\Phi_{1\not{2}}^{-1}\tilde{u}_1}dx^1\\
&&\quad \cdot \int_{U_{3n}\times U_y}\rho_{\not{x}^2}(x^1,x^3,\cdots , x^n,y)\\
&& \quad \times dx^1 dx^3\cdots dx^n dy\\
&=& \int_{\Phi_{1\not{2}}^{-1}\tilde{u}_1}\rho_1(x^1)dx^1\\
&=& \int_{{\phi_1}^{-1}\tilde{u}_1}\rho_1(x^1)dx^1
\end{eqnarray*}
where the last equality follows from the fact that $\phi_1$ and hence $\phi_1^{-1}$ is independent of $x_2$.

Again,
\begin{eqnarray*}
\int_{\tilde{u}_1} (\mathbb{P}\rho)_1(x^1)dx^1 &=& \int_{\tilde{u}_1\times U_2n\times U_y}\mathbb{P}\rho(x,y)dxdy\\
&=& \int_{\Phi^{-1}({\tilde{u}_1\times U_{2n}\times U_y})}\rho(x,y)dxdy\\
&=& \int_{\phi^{-1}\tilde{u}_1\times U_{2n}\times U_y}\rho(x,y)dxdy\\
&=& \int_{\phi_1^{-1}\tilde{u}_1}dx^1\int_{U_{2n}\times U_y}\rho(x)dx^2\cdots dx^ndy\\
&=& \int_{\phi_1^{-1}\tilde{u}_1}\rho_1(x^1)dx^1.
\end{eqnarray*}

Hence, 
\begin{eqnarray}
\int_{\tilde{u}_1}(\mathbb{P}\rho_{\not{x}^2})_1(x^1)dx^1 = \int_{\tilde{u}_1} (\mathbb{P}\rho)_1(x^1)dx^1
\end{eqnarray}
for any $\tilde{u}_1\in U_1$. Hence, $(\mathbb{P}\rho_{\not{x}^2})_1=(\mathbb{P}\rho)_1$ almost everywhere, when $\phi_1$ is independent of $x^2$. 
Now, in \cite{liang_discrete} it was shown that
\begin{eqnarray*}
H_{\not{x}^2}(x_{t+1}^1) &=& -\int_U(\mathbb{P}\rho_{\not{x}^2})_1(y^1)\log((\mathbb{P}\rho_{\not{x}^2})_1(y^1))\\
&& \cdot \rho(x^2|x^1,x^3\cdots , x^n,y)\\
&& \cdot \rho_{3\cdots n,y}(x^3,\cdots x^n, y)dy^1dx^2dx^3\cdots dx^ndy\\
&=& -\int_U(\mathbb{P}\rho)_1(y^1)\log((\mathbb{P}\rho)_1(y^1))\\
&& \cdot \rho(x^2|x^1,x^3\cdots , x^n,y)\\
&& \cdot \rho_{3\cdots n,y}(x^3,\cdots x^n, y)dy^1dx^2dx^3\cdots dx^ndy\\
&=& - \int_{U_1}(\mathbb{P}\rho)_1(y^1)\log((\mathbb{P}\rho)_1(y^1))dy^1\\
&& \cdot \Big[\int_{U_2}\int_{U_{3n}\times U_y}\rho(x^2|x^1,x^3\cdots , x^n,y)\\
&& \cdot \rho_{3\cdots n,y}(x^3,\cdots x^n, y)dx^2dx^3\cdots dx^ndy\Big]\\
&=& - \int_{U_1}(\mathbb{P}\rho)_1(y^1)\log((\mathbb{P}\rho)_1(y^1))dy^1\\
&=& H(x_{t+1}^1)
\end{eqnarray*}
This is because the integrand in the brackets integrate to one. Hence, $H(x_{t+1}^1) = H_{\not{x}^2}(x_{t+1}^1)$, that is,
\begin{eqnarray*}
H(z_{t+1}^1) = H_{\not{z}^2}(z_{t+1}^1)
\end{eqnarray*}
Hence, $T_{z^2\to z^1} = 0$. 
\end{proof}

\subsection{Revisiting the examples}

In this subsection we revisit the examples presented in section (\ref{section_DI}) and show that the new definition of information overcome some of the issues with directed information. 

{\it Example \ref{example_zero_IT} revisited :}
For the three state example discussed in section \ref{section_DI}, we compute the information transfer using our proposed definition of information transfer and the formula (\ref{transferx1y}). In Fig. 4(a), we plot the information transfer over $n$ time steps and we notice that the information transfer from $x_1\to x_2$ and $x_1\to x_3$ are identically zero. 

%


{\it Example \ref{example_indirect} revisited :} We compute the information transfer from $x_1\to x_k$ and from $x_k\to x_N$ for $k=15$. Since, $x_1$ is connected to $x_k$ through series of $k$ delays, as expected the information transfer $T_{x_1\to x_k}$ is zero over $n<k$ time step and for $n=k$ the information transfer jumps from zero to non-zero value as shown in Fig. 4(b). Similarly information transfer from $x_k\to x_N$ also shows a sudden jump. This proves that our proposed definition of information transfer obeys the time constraints of transfer.


\begin{figure}[htp!]
\centering
\subfigure[]{\includegraphics[scale=.38]{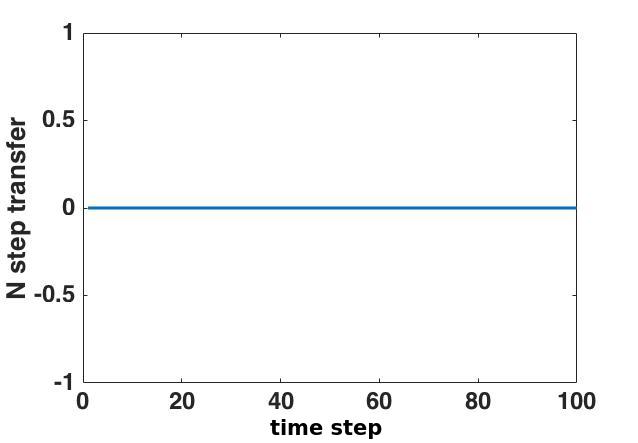}}\label{IT_3_node}
\subfigure[]{\includegraphics[scale=.38]{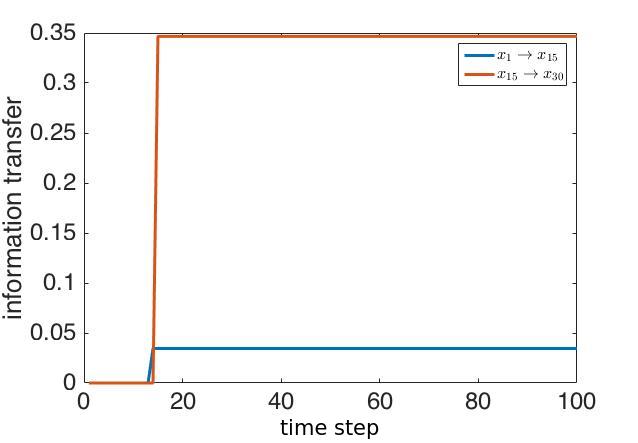}}\label{IT_chain_30_nodes}
\caption{(a) N-step information transfer for example 1. (b) N-step information transfer for example 2.}
\end{figure}

\section{Information Transfer in Linear Control Systems}\label{section_IT_CDS}
In this section, we derive expression for information transfer in control dynamical system. In particular, we derive expressions for information transfer from state to output, input to state, and input to output. For the simplicity of discussion we also restrict our discussion to single input single output case. 
 
We consider the following linear linear time invariant  system with input and output. 
\begin{eqnarray}
\begin{aligned}
& z(t+1) =  A z(t) + Bu(t)\\
& \theta(t) = C z(t) + \omega(t)
\end{aligned}
\end{eqnarray}

where $z\in\mathbb{R}^N$ are the states, $A\in\mathbb{R}^{N\times N}$, $B$ and $C$ are of appropriate dimensions, $u(t)\in\mathbb{R}$ is the input, $\theta\in\mathbb{R}$ is output and $\omega(t)$ is the output noise, which is assumed to be zero mean i.i.d. Gaussian noise. For the simplicity of presentation, we will restrict our discussion to the case where the state space $z$ is split in only two subspace i.e., $z=(x^\top,y^\top)^\top$ and the inputs and outputs are one dimensional. With this assumption, we have following splitting of the $A,B$, and $C$ matrices. 


\begin{eqnarray}
A = \begin{pmatrix}
A_x & A_{xy} \\ 
A_{yx} & A_y
\end{pmatrix}, B = \begin{pmatrix}
B_x\\
B_y
\end{pmatrix} , C = \begin{pmatrix}
C_x & C_y
\end{pmatrix}
\end{eqnarray}

\subsection{Information from Input to State}
The evolution of the state $x$ is $x' = A_xx + A_{xy}y + B_{x}u$. In deriving the formulas in this section, we think of the input $u(t)$ as a i.i.d. Gaussian variable such that
\begin{eqnarray*}
\textnormal{E}[u(t)u(t+k)^\top] = 0, \forall k \neq t \textnormal{ and }\textnormal {E}[zu^\top] = 0.
\end{eqnarray*} 

Hence,
\begin{eqnarray}
\Sigma_{x'|x} = A_{xy}\Sigma_x^sA_{xy}^\top + B_{x}\Sigma_u B_{x}^\top
\end{eqnarray}
where $\Sigma_x^s$ is the Schur complement of $\Sigma_x$ in $\Sigma_z$ and $\Sigma_u$ is the covariance of the input at time $t$.

When the input $u$ is held frozen, 
\begin{eqnarray}
\Sigma_{x'|x}^{\not{u}} = A_{xy}\Sigma_x^sA_{xy}^\top
\end{eqnarray}

Hence the information transfer from $u_1$ to $x$ is 
\begin{eqnarray}
T_{u\to x} = \frac{1}{2}\log\frac{|A_{xy}\Sigma_x^sA_{xy}^\top + B_{x}\Sigma_u B_{x}^\top|}{|A_{xy}\Sigma_x^sA_{xy}^\top|}
\end{eqnarray}
where $|\cdot|$ is the determinant.

\subsection{Information from State to Output} 
In this subsection, we look at the information transfer from the states of the system to the output of the control system. For simplicity, we only look at the information transfer from the entire state space $z$ to the output $\theta$. The general case of the information transfer from any state $z_i$ to any output $\theta_j$ for a MIMO system can be dealt with in similar manner.

From the output equation, we have
\begin{eqnarray*}
\Sigma_{\theta ' |\theta} = C\Sigma_{z'} C^\top + \Sigma_\omega -(CA\Sigma_zC^\top)(C\Sigma_zC^\top)^{-1}(CA\Sigma_zC^\top)^\top
\end{eqnarray*}
where $\Sigma_{z'}= A\Sigma_zA^\top + B\Sigma_uB^\top$ and $\Sigma_\omega$ is the covariance of the output noise.

When $z$ is frozen, the output equation is
\begin{eqnarray}
\theta_{\not{z}}(t) = \omega(t)
\end{eqnarray}

Hence, $\Sigma_{\theta '|\theta}^{\not{z}} = \Sigma_\omega$.

So the transfer is
\begin{eqnarray*}
T_{z\to \theta} = \frac{1}{2}\log\frac{|\Sigma_{\theta'|\theta}|}{|\Sigma_\omega|}
\end{eqnarray*}
where
\begin{eqnarray*}
\Sigma_{\theta'|\theta} = C\Sigma_{z'} C^\top + \Sigma_\omega-(CA\Sigma_zC^\top)(C\Sigma_zC^\top)^{-1}(CA\Sigma_zC^\top)^\top
\end{eqnarray*}

\subsection{Information from Input to Output}
As before, we address the case of SISO systems for simplicity.
We have 
\begin{eqnarray*}
\Sigma_\theta' = CA\Sigma_z A^\top C^\top + CB\Sigma_u B^\top C^\top
\end{eqnarray*}
Hence,
\begin{eqnarray}\label{output_covariance}
\Sigma_{\theta '|\theta} = CA\Sigma_z A^\top C^\top + CB\Sigma_u B^\top c^\top +\Sigma_\omega - (CA\Sigma_z C^\top)(C\Sigma_z C^\top)^{-1}(CA\Sigma_z C^\top)^\top
\end{eqnarray}
When $u$ is frozen, we have
\begin{eqnarray}
\begin{aligned}
& z'= A z\\
& \theta = Cz + \omega
\end{aligned}
\end{eqnarray}
Hence,
\begin{eqnarray}\label{output_covariance_u_frozen}
\Sigma_{\theta_{\not{u}}'|\theta_{\not{u}}} = CA\Sigma_z A^\top C^\top +\Sigma_\omega - (CA\Sigma_z C^\top)(C\Sigma_z C^\top)^{-1}(CA\Sigma_z C^\top)^\top
\end{eqnarray}

Hence the transfer is
\begin{eqnarray}
T_{u\to\theta} = \frac{1}{2}\log\frac{|\Sigma_{\theta '|\theta}|}{|\Sigma_{\theta_{\not{u}}'|\theta_{\not{u}}}|}
\end{eqnarray}
where $\Sigma_{\theta '|\theta}$ and $\Sigma_{\theta_{\not{u}}'|\theta_{\not{u}}}$ are given by equations (\ref{output_covariance}) and (\ref{output_covariance_u_frozen}) respectively.

\subsection{Information Transfer in Feedback Control Systems}
\begin{figure}
\centering
\includegraphics[scale=.55]{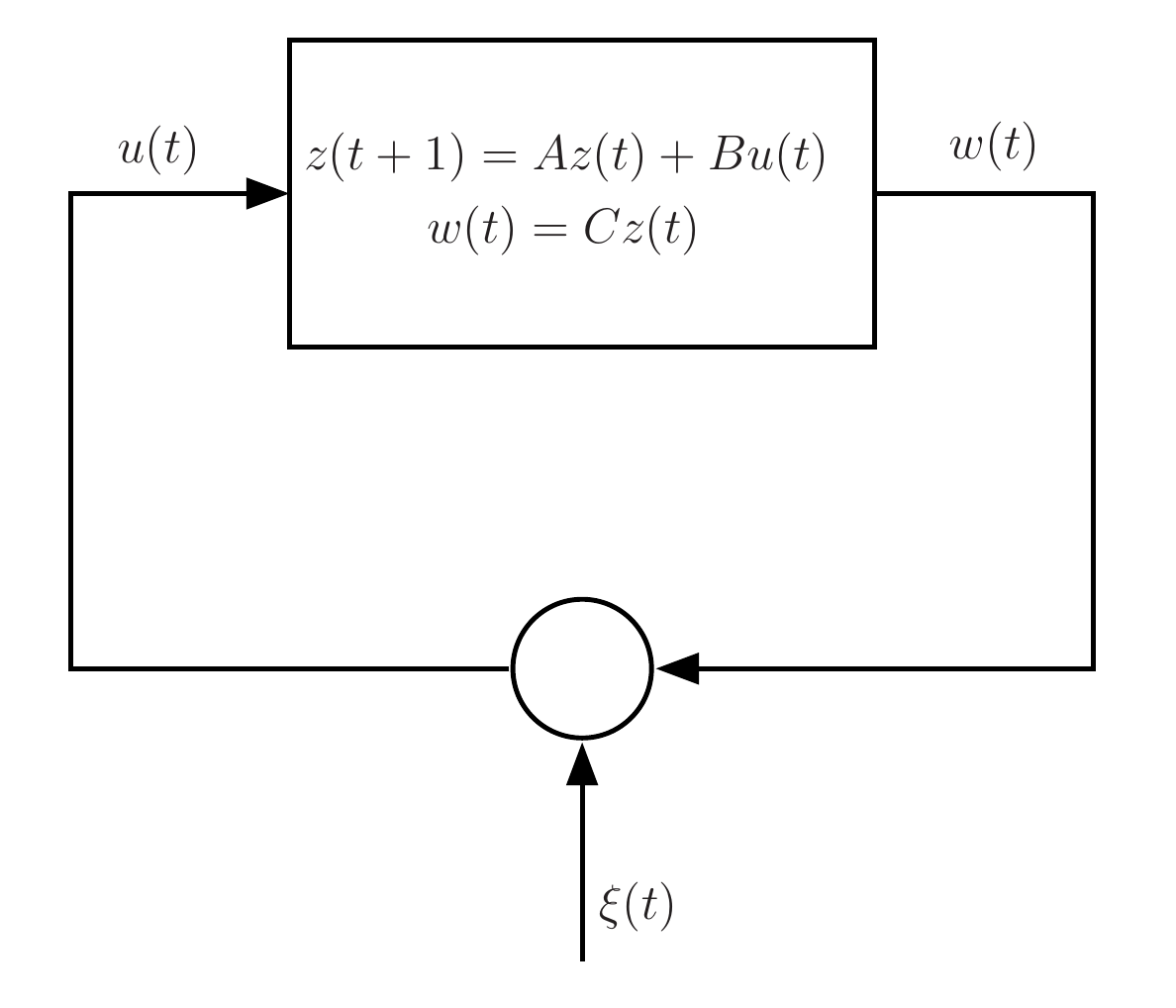}\label{feedback_fig}
\caption{Feedback control system}
\end{figure}

Information theory and feedback systems are interlinked and researchers have studied their interplay \cite{wong1999systems,tatikonda2004stochastic,li2013bode,tatikonda2004control}. Again, Massey \cite{IT_massey_directed} had shown that the feedback capacity $(C_{fb})$ satisfies 
\begin{eqnarray}
C_{fb} \leq \lim_{n\to\infty}\max_{\rho(X^n||Y^{n-1})}\frac{1}{n}I(X^n\to Y^n)
\end{eqnarray}
where $\rho(X^n||Y^{n-1})$ is the conditional distribution of $X^n$, causally conditioned on $Y^{n-1}$. Again, in \cite{elia2004bode}, it has been proved that in a feedback system, the Bode sensitivity transfer function from the output to the input is equal to the average directed information from the output $(w)$ to the input $(u)$, where the average directed information from the output to the input is defined as
\begin{eqnarray}
\lim_{n\to\infty}\frac{1}{n}I(w^n\to u^n)
\end{eqnarray}

In this section, we consider the following feedback control system (figure 5)
\begin{eqnarray}\label{feedback_system}
\begin{aligned}
& z(t+1)= Az(t)+Bu(t)\\
& w(t)=Cz(t),\;\;u(t)=w(t)+\xi(t)
\end{aligned}
\end{eqnarray}

and we are interested in computing information flow from output to input of the feedback control system i.e., $T_{w\to u}$. This is the information transfer from the plant to itself. Towards this goal we write the feedback control system as follows:
\begin{eqnarray}
z(t+1)=(A+BC) z(t)+B\xi(t)\nonumber\\
w(t)=Cz(t),\;\;\;u(t)=w(t)+\xi(t)
\end{eqnarray}

Hence, $H_{\not{w}}(\rho(u^{'}|u))=H(\rho(\xi^{'}))$. With this, and theorem (\ref{n_step_IT}), next we show that the Bode sensitivity transfer function, $\mathcal{S}$, from $w$ to $u$ is same as the average information transfer from $w$ to $u$.

\begin{theorem}
\begin{eqnarray}\label{bode}\nonumber
\bar{T}_{x\to y} = \int_{-\frac{1}{2}}^{\frac{1}{2}}\log |\mathcal{S}(e^{j2\pi\theta})|d\theta = \sum_{i = 1}^m\log|\lambda_i(A_u)|
\end{eqnarray}
where $\lambda_i(A_u)$ are the unstable eigenvalues of the open loop system.
\end{theorem}

\begin{proof}
From theorem (\ref{n_step_IT}), we have
\begin{eqnarray*}
\begin{aligned}
&\bar{T}_{x\to y} = \lim_{T\to \infty} \frac{1}{T}\sum_{i=1}^T(T_{x\to y})_0^i \\
&=  \lim_{T\to \infty} \frac{1}{T} [H(u^T) - H_{\not{w}}(u^T)]= \lim_{T\to \infty} \frac{1}{T}[H(u^T) - H(\xi^T)]\\
&=\lim_{T\to \infty} \frac{1}{2T} \log \frac{|\Sigma_{u^T}|}{|\Sigma_{\xi^T}|} =  \int_{-\frac{1}{2}}^{\frac{1}{2}}\log \frac{|u(e^{j2\pi\theta})|^2}{|\xi(e^{j2\pi\theta})|^2}d\theta\\
&= \int_{-\frac{1}{2}}^{\frac{1}{2}}\log |\mathcal{S}(e^{j2\pi\theta})|d\theta = \sum_{i = 1}^m\log|\lambda_i(A_u)|
\end{aligned}
\end{eqnarray*}
where $\lambda_i(A_u)$ are the unstable poles of $A$.

\end{proof}

Hence, the average directed  information and average information transfer are the same for the above feedback control system. This theorem, along with the fact that the channel capacity of a feedback channel is related to directed information, presents further evidence that our formulation of information transfer, though developed from dynamical systems point of view, is consistent with already existing concepts in information theory and control theory. We feel this will allow us to study the concepts of information theory, control theory and dynamical systems from a more fundamental and common point of view.

\begin{example}
Consider the feedback system given in figure 5 with
\begin{eqnarray*}
\begin{aligned}
A = \left(\begin{array}{cc}
4 & 2 \\ 
0 & 0.7
\end{array} \right),\quad B = \left(\begin{array}{c}
1 \\ 
0
\end{array} \right), \quad C = [\begin{array}{cc}
-3.5 & 0
\end{array} ]
\end{aligned}
\end{eqnarray*}
We also assume that the noise $\omega$ is i.i.d. zero mean unit variance Gaussian white noise.
\begin{figure}[htp!]
\centering
\subfigure[]{\includegraphics[scale=.28]{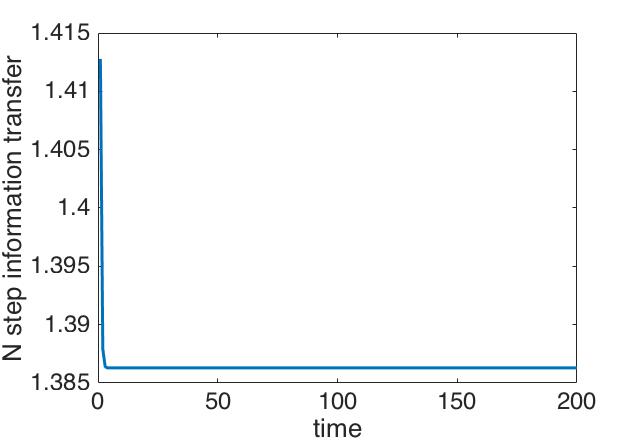}}
\subfigure[]{\includegraphics[scale=.28]{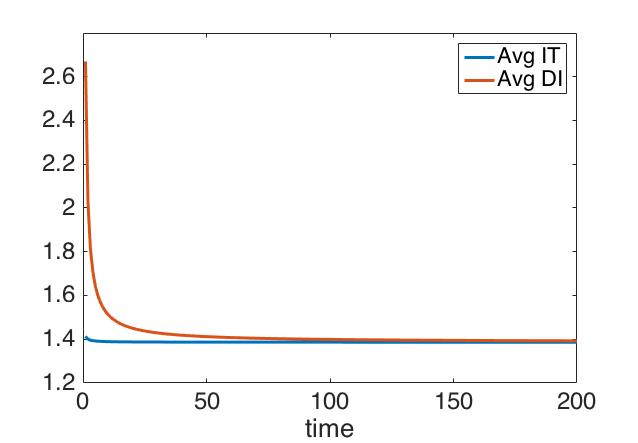}}
\caption{(a) N-step information transfer; (b) Average information transfer and average directed information}\label{avg_IT_avg_DI_feedback}
\end{figure}
The open loop system has unstable pole at $4$. Hence the Bode integral of the sensitivity transfer function is $\log(4)=1.3863$. Figure 6(a) shows the N-step information transfer from the output of the system to the input of the system. The steady state value of the information transfer is $1.3863$. Similarly, the average information transfer from the output to the input is also $1.3863$. This is shown in figure 6(b). In this figure we also plot the average directed information transfer from the output to input. Here again we observe that the average information transfer converges to the Bode integral of the sensitivity transfer function much faster than the average directed information.
\end{example}

\section{Information Transfer and Structural Controllability and Observability}\label{section_control}
Structural controllability was developed by \cite{lin1974structural}, and is a weaker notion of controllability, where a system is said to structurally controllable, if either the system is controllable or is controllable when the non-zero entries of the system matrix are perturbed so that the perturbed system becomes controllable. Associated with a system
\begin{eqnarray}\label{struc_control}
& z(t+1) = Az(t) + \xi (t), \quad z\in\mathbb{R}^N
\end{eqnarray}
is a directed graph $Z=(V_z,E_z)$ where $V_z = \{z_i|i = 1,\hdots ,N \}$ is the node set and $E_z$ is the edge set, such that there is directed edge from $z_i$ to $z_j$ iff ${ji}^{th}$ entry of $A$ is non-zero. Now if a single control input is placed at node $k$, then one can construct an extended graph $G=(V,E)$, with one extra node $u$ and one extra edge from $u$ to $z_k$. With this, the system is said to structurally controllable if all the nodes of the extended graph can be reached from node $u$ \cite{siljak2011decentralized}. Since $u$ is directly connected to $z_k$, one can say that the system, with input at $z_k$, is structurally controllable if there exists paths to all the nodes of $Z$ from $z_k$.

In this section we show that if information transfer from the input to all the nodes is non-zero, then the system is structurally controllable. Note that in this section, by information transfer, we will mean $n$-step information transfer for $n\in\mathbb{Z}_{>0}$.

\begin{theorem}\label{n_step_IT_path}
If the $k$-step information transfer from state $z_i$ to $z_j$ of system (\ref{struc_control}) is non-zero for some $k\in\mathbb{Z}_{>0}$, then there is a directed path from node $z_i$ to $z_j$ in the corresponding directed graph for the system.
\end{theorem}

\begin{proof}
We prove this by contrapositive argument. Without loss of generality, we assume $i=2$ and $j=1$. Suppose there is no path from $z_2$ to $z_1$. Hence, $[A^k]_{12} = 0$ for all $k\in\mathbb{Z}$ \footnote{Here  $[A^k]_{ij}$ means the $ij^{th}$ entry of $A^k$ }. We show that in this case the $k$-step information transfer from $z_2$ to $z_1$ is zero for all $k\in\mathbb{Z}$. The evolution of the states can be written as
\begin{eqnarray}
z(k) = A^kz(0) + A^{k-1}\xi (0) + \hdots + \xi (k-1)
\end{eqnarray}

Hence,
\begin{eqnarray}\label{z_1_dynamics}\nonumber
z_1(k) &=& \sum_{j=1, j\neq 2}^N [A^k]_{(1j)}z_j(0) + \sum_{j=1, j\neq 2}^N [A^{k-1}]_{(1j)}\xi_j(0)\\
&& + \hdots + \xi_1 (k)
\end{eqnarray}
since $[A^k]_{12}=0$ and when $z_2$ is held frozen, we have
\begin{eqnarray}\nonumber\label{z_1_dynamics_z_2_frozen}
z_{1\not{2}}(k) &=& \sum_{j=1, j\neq 2}^N [A_{\not{2}}^k]_{(1j)}z_j(0) + \sum_{j=1, j\neq 2}^N [A_{\not{2}}^{k-1}]_{(1j)}\xi_j(0)\\
&& + \hdots + \xi_1 (k)
\end{eqnarray}
where $A_{\not{2}}$ is the system matrix when $z_2$ is held frozen, that is $A_{\not{2}}$ is obtained from $A$ by deleting the second row and second column of $A$.

Now,
\begin{eqnarray}\nonumber
&&[A^2]_{(1j)} = \sum_{i=1}^Na_{1i}a_{ij} = a_{12}a_{2j} + \sum_{i\neq 2}a_{1i}a_{ij} \\
&&= \sum_{i\neq 2}a_{1i}a_{ij}=[A_{\not{2}}^2]_{1j}
\end{eqnarray}
since $a_{12} = 0$.

Similarly,
\begin{eqnarray*}
&&[A^3]_{1j} = [A^2A]_{1j} = [A_2A]_{1j} = \sum_i (a_2)_{1i}a_{ij} =\\
&& (a_2)_{12}a_{2j} + \sum_{i\neq 2}(a_2)_{1i}a_{ij} = \sum_{i\neq 2}(a_2)_{1i}a_{ij}=[A_{\not{2}}^3]_{1j}
\end{eqnarray*}
since $[A^2]_{12} = (a_2)_{12} = 0$.
Let $[A^l]_{1j} = [A_{\not{2}}^l]_{1j}$ for some $l\in\mathbb{Z}_{>0}$. Then
\begin{eqnarray*}
&&[A^{l+1}]_{1j} = [A^lA]_{1j} = [A_lA]_{1j} = \sum_i (a_l)_{1i}a_{ij} =\\
&& (a_l)_{12}a_{2j} + \sum_{i\neq 2}(a_l)_{1i}a_{ij} = \sum_{i\neq 2}(a_l)_{1i}a_{ij}=[A_{\not{2}}^{l+1}]_{1j}
\end{eqnarray*}
Hence, $[A^{k}]_{1j}=[A_{\not{2}}^{k}]_{1j}$ for all $k\in\mathbb{Z}$. Hence from (\ref{z_1_dynamics}) and (\ref{z_1_dynamics_z_2_frozen}), $z_1(k) = z_{1\not{2}}(k)$.
So,
\begin{eqnarray*}
&&H(z_1(k)|z_1(k-1)\hdots z_1(0))\\
&&= H_{\not{z}_2}(z_1(k)|z_1(k-1)\hdots z_1(0)).
\end{eqnarray*}
Hence $[T_{z_2\to z_1}]_0^k=0$ for all $k\in\mathbb{Z}$. Hence, if there is non zero information transfer from $z_2$ to $z_1$, then there exists a directed path from $z_2$ to $z_1$.
\end{proof}

\begin{theorem}\label{structral_controllability}
If the information transfer from the input to all the states are non-zero, then the system is structurally controllable.
\end{theorem}
\begin{proof}
We consider the extended graph associated with the dynamical system and consider the input node as an extra state. From theorem (\ref{n_step_IT_path}), if the information transfer from the input to any state $z_i$ is non-zero, then there exists a path from the input node to the node $z_i$ in the directed graph associated with the system. So, if the information transfer from the input to all the states is non zero, then there exists directed paths from the input node to all the other nodes in the directed graph associated with the system and hence the system is input reachable and hence is structurally controllable. 
\end{proof}

Results for structural observability follows from duality and can be stated as
\begin{theorem}\label{structural_observability}
If the information transfer from all the states to the output is non zero, then the system is structurally observable.
\end{theorem}

Here we have proved the result for structural controllability for SISO case, but the result for MIMO case is similar.

\section{Conclusion}\label{section_conclusion}
In this paper, we discussed how directed information fails to capture the causal structure in a dynamical system and provided a new measure of causality based on a new definition of information transfer and have shown that it captures the intuitions of causality. The new definition of information, which is based on dynamical system setting, has been generalized to define information transfer between the various signals in a control dynamical system. This generalization resulted in connecting the information transfer with Bode integral of the sensitivity transfer function from the output to the input in a feedback control system.

\bibliographystyle{IEEEtran}
\bibliography{ref1}

\begin{thebibliography}{10}
\providecommand{\url}[1]{#1}
\csname url@samestyle\endcsname
\providecommand{\newblock}{\relax}
\providecommand{\bibinfo}[2]{#2}
\providecommand{\BIBentrySTDinterwordspacing}{\spaceskip=0pt\relax}
\providecommand{\BIBentryALTinterwordstretchfactor}{4}
\providecommand{\BIBentryALTinterwordspacing}{\spaceskip=\fontdimen2\font plus
\BIBentryALTinterwordstretchfactor\fontdimen3\font minus
  \fontdimen4\font\relax}
\providecommand{\BIBforeignlanguage}[2]{{%
\expandafter\ifx\csname l@#1\endcsname\relax
\typeout{** WARNING: IEEEtran.bst: No hyphenation pattern has been}%
\typeout{** loaded for the language `#1'. Using the pattern for}%
\typeout{** the default language instead.}%
\else
\language=\csname l@#1\endcsname
\fi
#2}}
\providecommand{\BIBdecl}{\relax}
\BIBdecl

\bibitem{IT_socialmedia}
G.~V. Steeg and A.~Galstyan, ``Information transfer in social media,'' in
  \emph{{Proceedings of the 21st international conference on World Wide Web}},
  New York, NY, 2012, pp. 509--518.

\bibitem{IT_brain}
O.~Sporns, \emph{The networks of the brain}.\hskip 1em plus 0.5em minus
  0.4em\relax {MIT Press}, 2010.

\bibitem{IT_bionetwork1}
D.~J.~S. A.~Rao, A. O.~Hero and J.~D. Engel, ``Motif discovery in
  tissue-specific regulatory sequences using directed information.'' in
  \emph{{EURASIP J. on Bioinformatics and Systems Biology}}, 2007.

\bibitem{IT_bionetwork2}
------, ``Inference of biologically relevant gene influence networks using the
  directed information criterion.'' in \emph{{In proc. ICASSP}}, Toulouse,
  France, 2006.

\bibitem{granger_economics}
{C.W.J. Granger}, ``{Investigating causal relations by econometrics models and
  cross-spectral methods},'' \emph{{Econometrica}}, vol.~37, no.~3, pp.
  424--438, 1969.

\bibitem{IT_economics}
{C. A. Sims}, ``{Money, income and causality},'' \emph{{American Economic
  Review}}, vol.~62, pp. 540--552, 1972.

\bibitem{granger_causality}
C.~W.~J. Granger, ``Testing for causality,'' \emph{{Journal of Economic
  Dynamics and Control}}, vol.~2, pp. 329--352, 1980.

\bibitem{seth2004information}
H.~Touchette and S.~Lloyd, ``Information-theoretic approach to the study of
  control systems,'' \emph{Physica A: Statistical Mechanics and its
  Applications}, vol. 331, no.~1, pp. 140--172, 2004.

\bibitem{IT_massey_directed}
J.~L. Massey, ``Causality, feedback and directed information.'' in \emph{{Proc.
  Intl. Symp. on Info. th. and its Applications}}, Waikiki, Hawai, USA, 1990.

\bibitem{IT_kramer_directedit}
G.~Kramer, ``Directed information for channels with feedback,'' in \emph{{PhD
  Thesis}}, Swiss Federal Institute of Technology Zurich, 1998.

\bibitem{IT_schreiber}
T.~Schreiber, ``Measuring information transfer,'' \emph{{Physical Review
  Letters}}, vol. 85, no. 2, pp. 461--464, July, 2000.

\bibitem{liang_kleeman_prl}
{X. S. Liang and R. Kleeman}, ``{Information transfer between dynamical system
  components},'' \emph{{Physical Review Letters}}, vol.~95, p. 244101, 2005.

\bibitem{liang_predictability}
X.~S. Liang, ``Local predictability and information flow in complex dynamical
  systems,'' \emph{Physica D}, vol. 248, pp. 1--15, 2013.

\bibitem{Liang_Kleeman}
X.~S. Liang and R.~Kleeman, ``{A rigorous formalism of information transfer
  between dynamical system components. I. Discrete mapping},'' \emph{{Physica
  D}}, vol. 231, pp. 1--9, 2007.

\bibitem{IT_pnas}
{A. J. Majda† and J. Harlim}, ``{Information transfer between subspace of
  complex dynamical systems},'' \emph{{PNAS}}, vol. 104, no.~23, pp.
  9558--9563, 2007.

\bibitem{cdc_IT}
S.~Sinha and U.~Vaidya, ``Formalism for information transfer in dyamical
  networks,'' \emph{54th IEEE Conference on Decision and Control, Osaka,
  Japan}, pp. 5731 -- 5736, 2015.

\bibitem{Marko}
H.~Marko, ``The bidirectional communication theory- a generalization. of
  information theory,'' \emph{{IEEE Transaction on Communications}}, vol. Com -
  21, no. 12, pp. 1345--1351, Dec, 1973.

\bibitem{costanzo2014survey}
J.~A. Costanzo and J.~Dunstan, ``A survey of causality and directed
  information,'' 2014.

\bibitem{wong1999systems}
W.~S. Wong and R.~W. Brockett, ``Systems with finite communication bandwidth
  constraints. ii. stabilization with limited information feedback,''
  \emph{Automatic Control, IEEE Transactions on}, vol.~44, no.~5, pp.
  1049--1053, 1999.

\bibitem{tatikonda2004stochastic}
S.~Tatikonda, A.~Sahai, and S.~Mitter, ``Stochastic linear control over a
  communication channel,'' \emph{Automatic Control, IEEE Transactions on},
  vol.~49, no.~9, pp. 1549--1561, 2004.

\bibitem{li2013bode}
D.~Li and N.~Hovakimyan, ``Bode-like integral for continuous-time closed-loop
  systems in the presence of limited information,'' \emph{Automatic Control,
  IEEE Transactions on}, vol.~58, no.~6, pp. 1457--1469, 2013.

\bibitem{tatikonda2004control}
S.~Tatikonda and S.~Mitter, ``Control over noisy channels,'' \emph{Automatic
  Control, IEEE Transactions on}, vol.~49, no.~7, pp. 1196--1201, 2004.

\bibitem{elia2004bode}
N.~Elia, ``When bode meets shannon: Control-oriented feedback communication
  schemes,'' \emph{Automatic Control, IEEE Transactions on}, vol.~49, no.~9,
  pp. 1477--1488, 2004.

\bibitem{lin1974structural}
C.~T. Lin, ``Structural controllability,'' \emph{Automatic Control, IEEE
  Transactions on}, vol.~19, no.~3, pp. 201--208, 1974.

\bibitem{siljak2011decentralized}
D.~D. Siljak, \emph{Decentralized control of complex systems}.\hskip 1em plus
  0.5em minus 0.4em\relax Courier Corporation, 2011.

\end{thebibliography}

\end{document}